\documentclass[11pt,fleqn]{article}
\usepackage{amssymb}
\usepackage{amsfonts}
\usepackage{amsmath}
\usepackage{amsthm}
\usepackage{graphicx}
\usepackage{epsfig}
\usepackage{psfrag}
\usepackage{color}
\usepackage{dsfont}
\makeatletter
\xdef\@endgadget#1{{\unskip\nobreak\hfil\penalty50\hskip1em\hbox{}\nobreak
    \hfil#1\parfillskip=0pt\finalhyphendemerits=0\par}}
\def\@qedsymbol{${}_\blacksquare$}
\def\qed{\@endgadget{\@qedsymbol}}
\newtheorem{lemma}{Lemma}[section]

\newtheorem{example}[lemma]{Example}
\newtheorem{definition}[lemma]{Definition}

\newtheorem{proposition}[lemma]{Proposition}
\newtheorem{remark}[lemma]{Remark}
\newcommand{\mR}{\mathbb{R}}

\newcommand{\pperp}{\perp \!\!\!\perp}
\newcommand{\bq}{\begin{equation}}
\newcommand{\eq}{\end{equation}}
\DeclareMathOperator{\im}{im} 
\DeclareMathOperator{\rank}{rank}

\newcommand{\X}{\mathcal{X}}
\newcommand{\D}{\mathcal{D}}
\newcommand{\M}{\mathcal{M}}
\newcommand{\Z}{\mathcal{Z}}
\renewcommand{\L}{\mathcal{L}}

\def\BibTeX{{\rm B\kern-.05em{\sc i\kern-.025em b}\kern-.08em
    T\kern-.1667em\lower.7ex\hbox{E}\kern-.125emX}}

\title{\LARGE \bf Generalized Port-Hamiltonian DAE Systems\footnote{Submitted for publication.}}

\author{Arjan van der Schaft$^{1}$ and Bernhard Maschke$^{2}$
\thanks{$^{1}$Arjan van der Schaft is with the Bernoulli Institute for Mathematics, CS and AI,
        University of Groningen,  Groningen, the Netherlands, 
        {\tt\small a.j.van.der.schaft@rug.nl}}%
\thanks{$^{2}$Bernhard Maschke is with the Universit\'e Claude Bernard Lyon 1, CNRS, LAGEP, France, 
{\tt\small bernhard.maschke@univ-lyon1.fr}}%
}

\begin{document}

\maketitle
\thispagestyle{empty}
\pagestyle{empty}

\begin{abstract}
Motivated by recent work in this area we expand on a generalization of port-Hamiltonian systems that is obtained by replacing the Hamiltonian function representing energy storage by a general Lagrangian subspace. This leads to a new class of algebraic constraints in physical systems modeling, and to an interesting class of DAE systems. It is shown how constant Dirac structures and Lagrangian subspaces allow for similar representations, and how this leads to descriptions of the DAE systems entailing generalized Lagrange multipliers.
\end{abstract}

%

\section{Introduction}
It is well-known \cite{gsbm, phDAE, passivitybook} that port-Hamiltonian system dynamics may exhibit {\it algebraic constraints} in the state variables, leading to mixtures of {\it differential and algebraic equations} (DAEs). From a network modeling perspective algebraic constraints arise from interconnection of the subsystems composing the overall system. Within a standard port-Hamiltonian formulation the existence of such algebraic constraints is reflected in the properties of the underlying {\it Dirac structure} of the system. This Dirac structure is determined by the composition of Dirac structures of the subsystems, and needs not be a mapping from the effort variables to the flow variables, but instead a {\it relation} between them; see e.g. \cite{dalsmo, passivitybook, NOW} for more details. In this latter case, there are constraints between the co-energy variables of the system, which, via the Hamiltonian function, translate into algebraic constraints in the (energy) state variables. Examples include kinematic constraints in mechanical systems, and voltage or current constraints in electrical circuits.

On the other hand it was recently observed in \cite{bmxz}, see also the subsequent work \cite{mmw, mms, bmd, scholz}, that by {\it generalizing} the definition of linear port-Hamiltonian systems algebraic constraints may arise in {\it different} ways as well. At the same time in \cite{barbero}, motivated primarily by considerations in the geometric formulation of Lagrangian systems, systems with kinematic constraints, and optimal control, the definition of port-Hamiltonian systems was generalized by replacing the gradient of the Hamiltonian function in the port-Hamiltonian dynamics by a Lagrangian submanifold which is {\it not} necessarily the graph of the gradient of a Hamiltonian. This leads to algebraic constraints in the state variables which are of a different nature than those originating from effort constraints corresponding to Dirac structures.

In the present paper we will elaborate on the algebraic constraints of generalized port-Hamiltonian systems defined by Dirac structures as well as by Lagrangian subspaces; thus elucidating and complementing earlier contributions. For simplicity of exposition we will concentrate on {\it linear time-invariant finite-dimensional systems}, and moreover on the {\it lossless} case (no energy-dissipation) without external variables (inputs/outputs). For developments concerning {\it time-varying} linear generalized pH DAE systems or infinite-dimensional linear pH DAE systems we refer to \cite{mmw, bmxz}, and for the nonlinear case to \cite{barbero}. 

Conceptually, the current paper is closest to \cite{barbero} by emphasizing the geometric definition of generalized port-Hamiltonian systems as pairs of a Dirac structure and a Lagrangian subspace, while some constructions (as well as the emphasis on the linear case) are inspired by \cite{bmxz}. The paper is structured as follows. In Section II we give the geometric definition of linear generalized port-Hamiltonian DAE systems (without energy-dissipation and external variables), entailing algebraic constraints due to the Dirac structure as well as to the Lagrangian subspace. Inspired by \cite{bmxz} we give an explicit coordinate representation in terms of a parametrizing state vector. The end of Section II provides a number of simple, but illustrative, examples of algebraic constraints corresponding to either the Dirac structure or to the Lagrangian subspace. In Section III we zoom in on algebraic constraints and the underlying geometry of Dirac structures and Lagrangian subspaces. We show, and illustrate by an example, how algebraic constraints may be resolved by the use of extended state spaces, corresponding to the introduction of generalized Lagrange multipliers. Section V contains the conclusions and hints for extensions.

\section{Definition of generalized port-Hamiltonian DAE systems}
An {\it unconstrained} linear lossless port-Hamiltonian system without external variables on an $n$-dimensional linear state space $\mathcal{X}$ is described by the system of ordinary differential equations (ODEs)
\bq
\label{PH}
\dot{x} = JQx,
\eq
where $J: \X^* \to \X, J=-J^T,$ is a skew-symmetric mapping (also called Poisson structure map), and the symmetric matrix $Q=Q^T$ defines a {\it Hamiltonian function} $H(x)=\frac{1}{2}x^TQx$. 
Obviously by skew-symmetry of $J$
\bq
\label{EC}
\frac{d}{dt}H(x)= x^TQJQx=0,
\eq
expressing energy-conservation. 
On the other hand, in network modeling of physical systems, the dynamics is {\it not} always in ODE form \eqref{PH}, but instead involves {\it algebraic equations} in the state vector $x$. This was formalized in the standard definition of a port-Hamiltonian system by generalizing the skew-symmetric map $J$ to a general {\it Dirac structure}, defined as follows \cite{vanderschaftmaschkearchive, dalsmo, passivitybook}. Consider the product $\X \times \X^*$, with projections $\pi:  \X \times \X^* \to \X$ and $\pi^*:  \X \times \X^* \to \X^*$. Define on $\X \times \X^*$ the bilinear form
\bq
\label{+}
\langle (f_1,e_1), (f_2,e_2) \rangle_+ := 
<e_1 \mid f_2> +  <e_2 \mid f_1>,
\eq
with $(f_i,e_i) \in \X \times \X^*, i=1,2$ and $<e \mid f>$ denoting the duality product between $f \in \X$ and $e \in \X^*$.
\begin{definition}[Constant Dirac structure \cite{courant}]
A Dirac structure is a subspace $\D \subset \X \times \X^*$ such that $\D = \D^{\pperp_+}$, where ${}^{\pperp_+}$ denotes the orthogonal companion with respect to the bilinear form $\langle \cdot, \cdot \rangle_+$.
\end{definition} 
\begin{remark} \cite{NOW, passivitybook} An equivalent definition of a Dirac structure can be stated as follows. A subspace $\D \subset \X \times \X^*$ is a Dirac structure iff $\langle \cdot, \cdot \rangle_+$ restricted to $\D$ is {\it zero}, and $\D$ is {\it maximal} with respect to this property. The dimension of any Dirac structure $\D \subset \X \times \X^*$ is equal to $\dim \X$. Furthermore, by taking $f_1=f_2=f, e_1=e_2=e$ in \eqref{+} it follows that $<e \mid f>=0$ for any $(f,e) \in \D$, expressing power conservation, and generalizing skew-symmetry.
\end{remark}
A linear {\it port-Hamiltonian DAE system} with Hamiltonian $H(x)=\frac{1}{2}x^TQx$, briefly pH DAE system, is now geometrically given as\footnote{Substitute $f=-\dot{x},e=Qx$. The minus sign in $f=-\dot{x}$ ensures consistent power flow sign convention.}
\bq
\label{dynPH}
(-\dot{x},Qx) =(f,e) \in \D
\eq
Note that the graph of any skew-symmetric map $-J$
\bq
\label{Dgraph}
\D_J := \{(f=-Je,e) \in \X \times \X^* \mid e \in \X^* \}
\eq
is a special type of Dirac structure. In fact, a Dirac structure $\D$ can be represented into the form \eqref{Dgraph} for some skew-symmetric $J$ if and only if $\pi^* (\D) = \X^*$. On the other hand, if $\pi^* (\D) \neq \X^*$ then the dynamics \eqref{dynPH} gives rise to the algebraic constraints in the state $x$ given as
\bq
e=Qx \in \pi^* (\D)
\eq
This type of algebraic constraints will be referred to as {\it Dirac algebraic constraints}. They arise as constraints on the variables $e=\nabla H(x)$, called in port-based modeling terminology the {\it co-energy} variables (as opposed to the {\it energy} variables $x$; i.e., the variables in which the energy is expressed). Through the specification of the Hamiltonian $H(x)=\frac{1}{2}x^TQx$ they translate into algebraic constraints on the energy variables $x$ given as $Qx \in \D$.
A Dirac structure $\D$ with $\pi^* (\D) \neq \X^*$ will be referred to as a {\it singular} Dirac structure, and if $\pi^* (\D) = \X^*$ as a {\it regular} Dirac structure.

\medskip

Recently, and from different points of view \cite{bmxz, barbero}, it was noted that a {\it second type} of algebraic constraints can be formulated by generalizing the {\it Hamiltonian function} $H(x)=\frac{1}{2}x^TQx$ to a {\it Lagrangian subspace} of $\X \times \X^*$. This latter notion is defined as follows, resembling\footnote{It should be noted that the definitions of Lagrangian subspaces and Dirac structures {\it diverge} in the nonlinear case, with Dirac structures on a manifold $\X$ still defining pointwise a linear subspace of the product $T_x\X \times T_x^*\X, x \in \X$, while Lagrangian subspaces generalize to Lagrangian {\it submanifolds} of the cotangent bundle $T^*\X$.} the previous definition of a Dirac structure. Consider on $\X \times \X^*$ the alternate bilinear form
\bq
\langle (x_1,e_1), (x_2,e_2) \rangle_- := <e_1 \mid x_2> -  <e_2 \mid x_1>,
\eq
with $(x_i,e_i) \in \X \times \X^*, i=1,2$.
\begin{definition}[Lagrangian subspace]
A Lagrangian subspace is a subspace $\L \subset \X \times \X^*$ such that $\L = \L^{\pperp_-}$, where ${}^{\pperp_-}$ denotes the orthogonal companion with respect to the bilinear form $\langle \cdot, \cdot \rangle_-$.
\end{definition} 
\begin{remark} Alternatively, a Lagrangian subspace is defined as a {\it maximal} subspace $\L \subset \X \times \X^*$ on which $\langle \cdot, \cdot \rangle_-$ is zero. Similarly to Dirac structures, the dimension of any Lagrangian subspace $\L \subset \X \times \X^*$ is equal to $n=\dim \X$. 
\end{remark}
Note that the gradient of the Hamiltonian $H(x)=\frac{1}{2}x^TQx$ defines the special type of Lagrangian subspace
\bq
\label{Lgraph}
\L_Q := \{(x, Qx) \in \X \times \X^* \mid x \in \X \},
\eq
i.e., the graph of the symmetric mapping $Q$.
Furthermore, a Lagrangian subspace $\L$ can be put into the form \eqref{Lgraph} for a certain symmetric $Q$ if and only if $\pi (\L) = \X$, while if $\pi (\L) \neq  \X$ then the following algebraic constraints in the state $x$ arise
\bq
x \in \pi (\L)
\eq
This type of algebraic constraints will be referred to as {\it Lagrange algebraic constraints}, since they are determined by the Lagrangian subspace $\L$. A Lagrangian subspace $\L$ with $\pi (\L) \neq \X$ will be referred to as a {\it singular} Lagrangian subspace, and if $\pi (\L) = \X$ as a {\it regular} Lagrangian subspace.
\begin{definition}[Generalized pH DAE system]
Consider a Dirac structure $\D \subset \X \times \X^*$ and a Lagrangian subspace $\L \subset \X \times \X^*$. This defines the {\it generalized port-Hamiltonian DAE system} (briefly, gpH DAE system)  $(\D,\L)$, with dynamics given by
\bq
\label{d}
(-\dot{x},e) \in \D, \; (x,e) \in \L
\eq
\end{definition}
Here \eqref{d} should be read as follows. Consider any $x \in \X$ for which there exist $e \in \X^*$ and $f \in \X$ such that $(x,e) \in \L$ and $(f,e) \in \D$. Then\footnote{This is the reason for distinguishing between $x \in \X$ and $f \in \X$; strictly speaking $f$ is in the {\it tangent space} to $\X$ at $x$, which however by linearity of $\X$ can be identified with $\X$.} minus the velocity $-\dot{x}$ is given as any such $f$.

A {\it coordinate representation} of the dynamics \eqref{d} of the gpH DAE system $(\D,\L)$ can be obtained as follows. 
As shown in \cite{dalsmo, courant}, any Dirac structure $\D \subset \X \times \X^*$ for an $n$-dimensional linear space $\X$ can be represented in kernel representation as
\bq
\label{dk}
\D=\{(f,e) \in \X \times \X^* \mid Kf + Le=0 \}
\eq
for $n \times n$ matrices $K, L$ satisfying
\bq
\label{dk1}
KL^T + LK^T =0, \; \rank \begin{bmatrix} K & L \end{bmatrix} =n
\eq
Analogously, see Appendix Proposition \ref{A} for a proof, any Lagrangian subspace admits a kernel representation
\bq
\label{kernel1}
\L=\{(x,e) \in \X \times \X^* \mid S^Tx - P^Te=0 \}
\eq
for $n \times n$ matrices $P,S$ satisfying
\bq
\label{sp1}
S^TP = P^TS , \; \rank \begin{bmatrix} S^T & P^T \end{bmatrix} =n
\eq
Equivalently, the Lagrangian subspace $\L$ can be represented in image representation as
\bq
\label{image}
\L=\{(x,e) \in \X \times \X^* \mid \exists z \in \Z=\mR^n \mbox{ s.t. }\begin{bmatrix} x \\ e \end{bmatrix} = \begin{bmatrix} P \\ S \end{bmatrix}z \}
\eq
It follows that the dynamics of the gpH DAE system defined by the pair $(\D, \L)$ is obtained by setting $f= -\dot{x}$ in \eqref{dk}, yielding $K\dot{x}=Le$ with $(x,e) \in \L$. Using the image representation \eqref{image} of $\L$ this implies the following DAE system in the parametrizing state vector $z \in \Z$
\bq
\label{pHDAE}
KP \dot{z} = LSz
\eq
In case of Lagrange algebraic constraints the matrix $P$ is not of full rank, inducing algebraic constraints in $z$, while in case of Dirac algebraic constraints the matrix $K$ is not of full rank; also inducing algebraic constraints.

A Hamiltonian function for the coordinate representation \eqref{pHDAE}, in terms of the parametrizing state vector $z$, is defined by
\bq
H(z) := \frac{1}{2} z^T S^TP z
\eq
(note that $S^TP=P^TS$ by \eqref{sp1}). In fact, along solutions of \eqref{pHDAE}
\bq
\frac{d}{dt}H(z)= z^T S^TP\dot{z} = e^T \dot{x}=0,
\eq
since $e^Tf=0$ for all $(f,e) \in \D$. 

Note that there is much freedom in the definition of the parametrizing state vector $z$. On the other hand, \eqref{image} shows that $z$ can be taken to be a mixture of the $x$ and $e$ variables; i.e., a mixture of {\it energy} and {\it co-energy} variables. This can be formalized as follows. Consider any Lagrangian subspace $\L \subset \X \times X^*$. Then, see Appendix Proposition \ref{C} for a proof, there always exists a sub-vector $x_1$ of $x \in \X$, and a complementary sub-vector $e_2 \in \X^*$, such that $\L$ is represented as
\bq
\L=\{(x,e) \in \X \times \X^* \mid \begin{bmatrix} e_1 \\ x_2 \end{bmatrix} = \widehat{Q} \begin{bmatrix} x_1 \\ e_2 \end{bmatrix} \}
\eq
Particular cases are $x_1=x$ and $e_2$ void, in which case $\widehat{Q}=Q$, or $e_2=e$ and $x_1$ void, in which case $\widehat{Q}=Q^{-1}$ if $Q$ is invertible, and the {\it co-energy function} $\frac{1}{2}e^TQ^{-1}e$ is the Legendre transform of $H(x)=\frac{1}{2}x^TQx$.

An alternative, and in some sense {\it dual}, coordinate representation of a generalized port-Hamiltonian DAE system $(\D,\L)$ can be obtained as follows.
Consider based on \eqref{dk} and \eqref{dk1} the {\it image} representation of $\D$ given as  
\bq
\D = \im \begin{bmatrix} L^T \\ K^T \end{bmatrix},
\eq
and the {\it kernel representation} \eqref{kernel1} of $\L$. Substitution of $-\dot{x}=f=L^Tv, e=K^Tv$, with $v$ an alternative parametrizing state vector, then leads to the DAEs
\bq
\label{pHDAEa}
S^TL^Tv + P^TK^T\dot{v}=0
\eq
By pre-multiplying \eqref{pHDAEa} by $z^T$, and performing integration by parts on the second term $z^TP^TK^T\dot{v}$, this results in the previously obtained coordinate expression \eqref{pHDAE} in the parametrizing state vector $z$. Thus \eqref{pHDAEa} can be considered as {\it a dual} (or {\it adjoint}) representation to \eqref{pHDAE}.

\subsection{Examples}
Dirac algebraic constraints of port-Hamiltonian systems arise from from the interconnection of subsystems. On the other hand, Lagrange algebraic constraints reflect degeneracies in the definition of energy-storage. This is illustrated by the following examples. The first two are standard examples of Dirac algebraic constraints, while the last three show how Lagrange algebraic constraints arise in physical systems modeling.

\begin{example}[Mechanical systems with kinematic constraints]
\label{ex1}
Consider a mechanical system with position coordinates $q \in \mR^n$, momenta $p=M\dot{q} \in \mR^n$, and mass matrix $M=M^T>0$, subject to constant
kinematic constraints $A^T\dot{q}=0$, where $A$ is an $n \times k$ matrix. Consider a Hamiltonian function $H(q,p)= \frac{1}{2}p^TM^{-1}p + \frac{1}{2}q^TKq$ with $K$ a matrix defining the elastic potential energy, specifying a Lagrangian subspace as in \eqref{Lgraph}. The Dirac structure $\D$ is given as
\[
\begin{array}{rcl}
\D & = &\{(f_q,f_p,e_q,e_p) \in \mR^{2n} \times \mR^{2n} \mid \exists \lambda \in \mR^k \mbox{ s.t. } \\[3mm]
&& \begin{bmatrix} f_q \\ f_p \end{bmatrix} =
\begin{bmatrix} 0_n & I_n \\ -I_n & 0_n \end{bmatrix} \begin{bmatrix} e_q \\ e_p \end{bmatrix}  + \begin{bmatrix} 0 \\ A \end{bmatrix} \lambda, \; A^T e_p=0 \}
\end{array}
\]
Substitution of $e_p=M^{-1}p$ leads to the algebraic constraints $A^T M^{-1}p=0$. Furthermore, $A\lambda$ is the vector of constraint forces.
\end{example}

\begin{example}[LC-circuits]
Dirac algebraic constraints are ubiquitous in electrical circuits. Concentrating on $LC$-circuits, such constraints arise in two ways. The first case corresponds to the occurrence of a cycle in the circuit graph whose edges only contain {\it capacitors}. By Kirchhoff's voltage law the sum of the voltages across these capacitors is identically zero, leading in the port-Hamiltonian formulation to an algebraic constraint between the charges of those capacitors. The second case corresponds to the existence of a node in the circuit graph whose adjacent edges only contain {\it inductors}. By Kirchhoff's current law the sum of the currents entering this node is equal to zero, thus leading to an algebraic constraint between the flux linkages of those inductors.
\end{example}

\begin{example}[Mass-spring system with zero mass]
Consider a mass-spring system with standard Hamiltonian
$\widehat{H}(q,p)= \frac{1}{2}kq^2 + \frac{p^2}{2m}$, with $m$ mass and $k$ the spring constant.
Now let $m$ converge to zero, leading to the constraint $p=0$. For $m\neq 0$ the graph defined by $\nabla \widehat{H}$ is given as
\[
\begin{bmatrix} q \\ p \\ e_q \\ e_p \end{bmatrix} =
\begin{bmatrix} 1 & 0 \\ 0 & m \\ k & 0 \\ 0 & 1 \end{bmatrix}
\begin{bmatrix} z_1 \\ z_2 \end{bmatrix}
\]
with $z_1=q$ the {\it position} of the mass (an energy variable), and $z_2=\frac{p}{m}$ its {\it velocity} (a co-energy variable). 
Taking the limit $m \to 0$ one obtains the degenerate Lagrangian subspace
\[
\begin{bmatrix} q \\ p \\ e_q \\ e_p \end{bmatrix} =
\begin{bmatrix} 1 & 0 \\ 0 & 0 \\ k & 0 \\ 0 & 1 \end{bmatrix}
\begin{bmatrix} z_1 \\ z_2 \end{bmatrix}
\]
Hence $z_1=q, z_2=e_p$, and we obtain the gpH DAE system
\[
\begin{bmatrix} 1 & 0 \\ 0 & 0 \end{bmatrix} 
\begin{bmatrix} \dot{z}_1 \\ \dot{z}_2 \end{bmatrix} =
\begin{bmatrix} 0 & 1 \\ - k & 0 \end{bmatrix} 
\begin{bmatrix} z_1 \\ z_2 \end{bmatrix} 
\]
with Hamiltonian $H(z_1,z_2) = \frac{1}{2} kz_1^2$. 
(Note that in this simple example the remaining system is trivial, and necessarily $z_1=q=0$ whenever $k \neq 0$.)
\end{example}

\begin{example}[Mechanical systems with strong constraining force]
\label{ex4}
Consider a two-dimensional mass-spring system with Hamiltonian $\widehat{H}\left(q_{1},q_{2},p_{1},p_{2}\right)$
\[
\widehat{H}=\frac{1}{2} k_{1}q_{1}^{2}+ \frac{1}{2} k_{12}\left(q_{2}-q_{1}\right)^{2}+\frac{1}{2m_1}p_{1}^{2}+ \frac{1}{2m_2}p_{2}^{2}
\]
being the series interconnection of two masses $m_1,m_2$ and two springs with spring constants $k_1, k_{12}$. This defines the Lagrangian
subspace given in image representation as
\[
\begin{bmatrix}
q_{1}\\
q_{2}\\
p_{1}\\
p_{2}\\
e_{q_{1}}\\
e_{q_{2}}\\
e_{p_{1}}\\
e_{p_{1}}
\end{bmatrix}
=\begin{bmatrix}
1 & 0 & 0 & 0\\
1 & \frac{1}{k_{12}} & 0 & 0\\
0 & 0 & 1 & 0\\
0 & 0 & 0 & 1\\
k_{1} & -1 & 0 & 0\\
0 & 1 & 0 & 0\\
0 & 0 & \frac{1}{m_1} & 0\\
0 & 0 & 0 & \frac{1}{m_2}
\end{bmatrix}
\begin{bmatrix}
z_{1}\\
z_{2}\\
z_{3}\\
z_{4}
\end{bmatrix},
\]
where we have chosen the parametrizing state vector $z$ as the following mixture of energy and co-energy variables:
\[
z_1=q_1, \; z_2=k_{12}(q_2-q_1), \: z_3=p_1, \: z_4=p_2
\]
(thus $z_2$ equals the elastic force of the second spring). The Hamiltonian expressed in the $z$-vector is given as
\[
\widehat{H}\left(z\right)=\frac{1}{2} k_{1}q_{1}^{2}+\frac{1}{2k_{12}}z_{2}^{2}+\frac{1}{2m_1}z_{3}^{2}+ \frac{1}{2m_2}z_{4}^{2}
\]
Letting $k_{12} \to \infty$ (corresponding to replacing the second spring by a {\it rigid rod}) yields the degenerate Lagrangian subspace
\[
\begin{bmatrix}
q_{1}\\
q_{2}\\
p_{1}\\
p_{2}\\
e_{q_{1}}\\
e_{q_{2}}\\
e_{p_{1}}\\
e_{p_{1}}
\end{bmatrix}=
\begin{bmatrix}
1 & 0 & 0 & 0\\
1 & 0 & 0 & 0\\
0 & 0 & 1 & 0\\
0 & 0 & 0 & 1\\
k_{1} & -1 & 0 & 0\\
0 & 1 & 0 & 0\\
0 & 0 & \frac{1}{m_1} & 0\\
0 & 0 & 0 & \frac{1}{m_2}
\end{bmatrix}
\begin{bmatrix}
z_{1}\\
z_{2}\\
z_{3}\\
z_{4}
\end{bmatrix}
\]
entailing the algebraic constraint\footnote{This can be called a {\it geometric constraint}, although the set-up is {\it different} from the standard approach to geometric constraints following from the {\it integration} of kinematic constraints $A^T\dot{q}=0$ as in Example \ref{ex1} to $A^Tq =c$, with the vector $c$ determined by the initial condition of the system.} $q_1=q_2$, and 
leads to the following gpH DAE system
\[
\begin{array}{rcl}
\begin{bmatrix}
1 & 0 & 0 & 0 \\
1 & 0 & 0 & 0 \\
0 & 0 & 1 & 0 \\
0 & 0 & 0 & 1
\end{bmatrix}
\begin{bmatrix}
\dot{z}_1 \\ \dot{z}_2 \\ \dot{z}_3 \\ \dot{z}_4 
\end{bmatrix}
& = &
\begin{bmatrix}
0 & 0 & 1 & 0 \\
0 & 0 & 0 & 1 \\
-1 & 0 & 0 & 0 \\
0 & -1 & 0 & 0
\end{bmatrix} \cdot \\
&&
\begin{bmatrix}
k_1 & -1 & 0 & 0 \\
0 & 1 & 0 & 0 \\
0 & 0 & \frac{1}{m_1} & 0 \\
0 & 0 & 0 & \frac{1}{m_2}
\end{bmatrix}
\begin{bmatrix}
{z}_1 \\ {z}_2 \\ {z}_3 \\ {z}_4 
\end{bmatrix}
\end{array}
\]
with algebraic constraint $\frac{z_3}{m_1} = \frac{z_4}{m_2}$ (equality of velocity of the first and the second mass, linked by a rigid rod). Note that $z_2$ (whose derivative does not appear in the DAE system) represents the constraint force exerted by the rigid rod on the masses $m_1$ and $m_2$ (with opposite sign).
\end{example}

\begin{example}[Ideal transformer]
An electrical transformer is a magnetic energy storage element consisting of two coils coupled by
a magnetic core. Its constitutive relations define a Lagrangian subspace
in kernel representation 
\[
\label{ex1}
S^{T}\begin{bmatrix}
\varphi_{1}\\
\varphi_{2}
\end{bmatrix}=
P^{T} \begin{bmatrix}
i_{1}\\
i_{2}
\end{bmatrix}
\]
in the magnetic fluxes $\varphi_1, \varphi_2$ and currents $i_{1}, i_2$ corresponding to the two coils.
Here $S=\left(\frac{\mathcal{R}_{m}}{N_{1}N_{2}}\right)I_{2}$, with $I_2$ the $2 \times 2$ identity matrix, and 
\[
P=\begin{bmatrix}
\frac{N_{1}}{N_{2}}\left(1+\frac{\mathcal{R}_{m}}{\mathcal{R}_{l1}}\right) & 1\\
1 & \frac{N_{2}}{N_{1}}\left(1+\frac{\mathcal{R}_{m}}{\mathcal{R}_{l2}}\right)
\end{bmatrix}
\]
with reluctances $\mathcal{R}_{l1}$, $\mathcal{R}_{l2}$ and $\mathcal{R}_{m}$,
and $N_{1},N_{2}$ the number of turns of the two coils. For an \emph{ideal} transformer, $\frac{\mathcal{R}_{m}}{\mathcal{R}_{li}}\rightarrow0$ for $i=1,2$, in which case the rank of the matrix $P$ drops from $2$ to $1$, leading to Lagrange algebraic constraints and the well-known transformer ratio.

Now connect the transformer at port $1$ to a capacitor $q=C\,v_{C}$ and at port $2$ to an inductor $\Phi =L\, i_L$.
Adding the constitutive relations of the capacitor and inductor one obtains the extended Lagrangian subspace $\mathcal{L}_{tot}$ with kernel representation
\[
P_{tot}=\textrm{diag} \left( 
\begin{bmatrix}
\frac{N_{1}}{N_{2}} & 1\\
1 & \frac{N_{2}}{N_{1}} \end{bmatrix}, \,
\begin{bmatrix} C & 0 \\ 0 & L \end{bmatrix} \right)
, \; S_{tot}=\textrm{diag}\left(S,\,I_2 \right)
\]
in the energy variables $\varphi_{1}, \varphi_{2}, q, \Phi,$ and co-energy variables $v_{1}, v_{2}, i_{C},v_L$.
The total Dirac structure $\D_{tot}$ of the system is given by the matrices $K_{tot}=I_4$ and  
\[
L_{tot}= \begin{bmatrix}
0 & 0 & 1 & 0 \\
0 & 0 & 0 & -1 \\
-1 & 0 & 0 & 0 \\
0 & 1 & 0 & 0 
\end{bmatrix}
\]
This defines a gpH DAE system given as in \eqref{pHDAE} or \eqref{pHDAEa}.
\end{example}

\section{Algebraic constraint representations}
In this section we further analyze Dirac and Lagrange algebraic constraints.
First we elaborate on different representations of them.

Consider a Dirac structure $\D \subset \X \times \X^*$. Denote as before by $\pi: \X \times \X^* \to \X$ the projection on $\X$, and by $\pi^*: \X \times \X^* \to \X^*$ the projection on $\X^*$. Then following \cite{courant} the bilinear form $\Delta$ on the subspace $\pi^*(\D) \subset \X^*$ given as
\bq
\Delta (\pi^*(v),\pi^*(w)) := <\pi^*(v) \mid \pi (w)>, \quad v,w \in \D
\eq
is well-defined and skew-symmetric.
Conversely, it can be shown that any skew-symmetric form on a subspace of $\X^*$ defines a Dirac structure $\D$. Thus Dirac structures are in one-to-one correspondence with skew-symmetric forms defined on subspaces of $\X^*$. Furthermore, it follows \cite{dalsmo} that any Dirac structure $\D \subset \X \times \X^*$ can be {\it embedded} into the graph of a skew-symmetric map 
on an extended state space as follows. Consider any $\D \subset \X \times \X^*$, and suppose $\pi^*(\D) \subset \X^*$ is $(n-k)$-dimensional. Define $\Lambda:= \mR^k$. Then there exists a full-rank $n \times k$ matrix $G$ and a skew-symmetric $n \times n$ matrix $J$ such that $\D$ is given as the set of all points $(f,e) \in \X \times \X^*$ satisfying for some $\lambda \in \Lambda$
\bq
\label{Diracconstraint}
\begin{array}{rcl}
-f &= &Je + G\lambda \\[2mm]
0 & = & G^Te
\end{array}
\eq
Conversely, any such equations for a skew-symmetric map $J: \X^* \to \X$ define a Dirac structure.
Hence Dirac structures correspond to {\it extended skew-symmetric maps}
\bq
\label{extskew}
\widetilde{J}= \begin{bmatrix} -J & -G \\ G^T & 0 \end{bmatrix} : \X^* \times \Lambda^* \to \X \times \Lambda
\eq
The subspace $\D \cap (0 \times \X^*)$ corresponds to the {\it conserved quantities} of the corresponding gpH DAE system \cite{NOW, passivitybook}, while $\pi^*(\D)$ defines the Dirac algebraic constraints.
%

Analogously, cf. Appendix Proposition \ref{B}, any Lagrangian subspace $\L \subset \X \times \X^*$ gives rise to the well-defined {\it symmetric} bilinear form on $\pi (\L)$
\bq
\Sigma (\pi(v),\pi(w)) := <\pi^*(v) \mid \pi (w)>, \quad v,w \in \L
\eq
Conversely any symmetric bilinear form on a subspace of $\X$ defines a Lagrangian subspace $\L$. Thus Lagrangian subspaces are in one-to-one correspondence with symmetric forms defined on subspaces of $\X$. 

Furthermore, analogously to the Dirac structure case, cf. Appendix Proposition \ref{B}, any Lagrangian subspace can be embedded into the graph of a symmetric mapping on an extended space. Indeed, for any Lagrangian subspace $\L$ there exists full-rank $n \times k$ matrix $M$ and a symmetric $n \times n$ matrix $Q$ such that $\L$ is given as the set of all points $(x,e) \in \X \times \X^*$ satisfying for some $\mu \in \M :=\mR^k$
\bq
\label{Lagrangeconstraint}
\begin{array}{rcl}
e &= & Qx + M\mu, \quad Q=Q^T \\[2mm]
0 & = & M^Tx
\end{array}
\eq
In fact, $M$ is such that $\ker M^T = \pi (\L)$. Hence Lagrangian subspaces correspond to {\it extended symmetric maps}
\bq
\label{extsym}
\widetilde{Q}= \begin{bmatrix} Q & M \\ M^T & 0 \end{bmatrix} : \X \times \M \to \X^* \times \M^*
\eq
Note that one can associate with \eqref{Lagrangeconstraint} the constrained optimization problem of extremizing $\frac{1}{2}x^TQx$ under the constraint $M^Tx=0$, or, using Lagrange multipliers, the unconstrained optimization
\bq
\mathrm{ext}_{x, \mu} \frac{1}{2}x^TQx + \mu^TM^Tx
\eq
with $\pi (\L)$ the constrained state space, and $\L \cap (\X \times 0)$ the set of constrained extrema.

These results can be employed for an extended representation of gpH DAE systems as follows. Consider a pH DAE system with only Dirac algebraic constraints, i.e., $P=I,S=Q,$ with Hamiltonian function $H(x)=\frac{1}{2}x^TQx$. Consider the representation of the Dirac algebraic constraints given in \eqref{Diracconstraint}, with extended state space $\X \times \Lambda$. Then define the singular Lagrangian subspace $\widetilde{\L} \subset \X \times \Lambda \times \X^* \times \Lambda^*$ specified by
\bq
\widetilde{P}:= \begin{bmatrix} I & 0 \\ 0 & 0 \end{bmatrix}, \; 
\widetilde{S}:= \begin{bmatrix} Q & 0 \\ 0 & I \end{bmatrix}
\eq
This corresponds to the parametrizing extended state vector $\tilde{z} = \begin{bmatrix} x \\ \lambda \end{bmatrix}$, and a Hamiltonian $\widetilde{H}(\tilde{z})$ given as
\bq
\widetilde{H}(\tilde{z}) = \frac{1}{2}\tilde{z}^T\widetilde{S}^T \widetilde{P}\tilde{z}= \frac{1}{2}x^TQx
\eq
(thus reducing in value to the original Hamiltonian function). The resulting gpH DAE system is given as
\bq
\label{DAEe}
\begin{bmatrix} I & 0 \\ 0 & 0 \end{bmatrix}
\begin{bmatrix} \dot{x} \\ \dot{\lambda} \end{bmatrix} =
\begin{bmatrix} J & G \\ -G^T & 0 \end{bmatrix}
\begin{bmatrix} Q & 0 \\ 0 & I \end{bmatrix}
\begin{bmatrix} x \\ \lambda \end{bmatrix} 
\eq
It is directly checked that any solution of \eqref{DAEe} projects to a solution of the original pH DAE system, and conversely any solution of the original ph DAE system is the projection of a solution of \eqref{DAEe}. Thus the pH DAE system with {\it singular} Dirac structure and {\it regular} Lagrangian subspace has been converted into a gpH DAE system in the extended state vector $\tilde{z}$ with {\it regular} Dirac structure but {\it singular} Lagrangian subspace. In this sense, {\it Dirac algebraic constraints} may be replaced by {\it Lagrange algebraic constraints}. This underlies some of the examples in \cite{bmxz}.

Dually, let us consider a gpH DAE system with only {\it Lagrange algebraic constraints}, with Dirac structure given by $K=I, L=J$. Consider the representation of the Lagrange algebraic constraints given in \eqref{Lagrangeconstraint}, i.e., 
\bq 
\widetilde{Q}:= \begin{bmatrix} Q & M \\ M^T & 0 \end{bmatrix} , \; \widetilde{P}:= \begin{bmatrix} I & 0 \\ 0 & I \end{bmatrix} 
\eq
This corresponds to a Hamiltonian $\widetilde{H}(\tilde{x})$ with $\tilde{x} = \begin{bmatrix} x \\ \mu \end{bmatrix}$ on the extended state space $\X \times \M$ given as
\bq
\widetilde{H}(\tilde{x}) = \frac{1}{2}x^TQx + x^TM\mu
\eq
Then define on this extended state space the singular Dirac structure $\widetilde{\D} \subset \X \times \M \times \X^* \times \M^*$ specified by
\bq
\widetilde{K}:= \begin{bmatrix} I & 0 \\ 0 & 0 \end{bmatrix}, \; 
\widetilde{L}:= \begin{bmatrix} J & 0 \\ 0 & I \end{bmatrix}
\eq
The resulting pH DAE system is given as
\bq
\begin{bmatrix} I & 0 \\ 0 & 0 \end{bmatrix}
\begin{bmatrix} \dot{x} \\ \dot{\mu} \end{bmatrix} =
\begin{bmatrix} J & 0 \\ 0 & I \end{bmatrix}
\begin{bmatrix} Q & M \\ M^T & 0 \end{bmatrix}
\begin{bmatrix} x \\ \mu \end{bmatrix} 
\eq
Thus dually we converted the gpH DAE system with {\it regular} Dirac structure and {\it singular} Lagrangian subspace into a pH DAE system in the extended state vector $\tilde{x}$ with {\it singular} Dirac structure and {\it regular} Lagrangian subspace.

The two above extensions may be {\it combined} as follows. Consider the representation of the Dirac algebraic constraints given in \eqref{Diracconstraint}, with extended state space $\X \times \Lambda$. Furthermore, consider the representation of the Lagrange algebraic constraints given in \eqref{Lagrangeconstraint}, with extended state space $\X \times \M$. Their combination yields the total extended space $\X \times \Lambda \times \M$, and the extended gpH DAE system
\bq
\label{merged}
\begin{bmatrix} I & 0 &0  \\ 0 & 0 & 0 \\ 0 & 0 & 0 \end{bmatrix}
\begin{bmatrix} \dot{x} \\ \dot{\lambda} \\ \dot{\mu} \end{bmatrix} =
\begin{bmatrix} J & G & 0 \\ -G^T & 0  & 0 \\ 0 & 0 & I \end{bmatrix}
\begin{bmatrix} Q & 0 & M \\ 0 & I & 0 \\ M^T & 0 & 0 \end{bmatrix}
\begin{bmatrix} x \\ \lambda \\\mu \end{bmatrix} 
\eq
The inclusion of the 'Lagrange multipliers' $\lambda$ and $\mu$ is especially useful for {\it simulation} of gpH DAE systems.

\begin{example}
Consider the system in Example \ref{ex4}, where we additionally impose as in Example \ref{ex1} the kinematic constraint $\dot{q}_1=0$.
The extended skew-symmetric map $\widetilde{J}$ as in \eqref{extskew} is simply given as
\bq
\widetilde{J}= \begin{bmatrix} 0 & 0 & 1 & 0 & 0 \\ 0 & 0 & 0 & 1 & 0 \\ -1 & 0 & 0 & 0 & 1\\ 0 & -1 & 0 & 0 & 0 \\
0 & 0 & -1 & 0 & 0 \end{bmatrix},
\eq
with $G^T= \begin{bmatrix}0 & 0 & 1 & 0 & 0 \end{bmatrix}$ and scalar Lagrange multiplier $\lambda$ corresponding to the constraint force for the kinematic constraint $\dot{q}_1=0$. The extended symmetric map $\widetilde{Q}$ as in \eqref{extsym} is given as
\bq
\widetilde{Q}= \begin{bmatrix} k_1 & -1 & 0 & 0 & 1 \\ -1 & 0 & 0 & 0 & -1 \\ 0 & 0 & \frac{1}{m_1} & 0 & 0\\ 0 & 0 & 0 & \frac{1}{m_2}  & 0 
\\
1 & -1 & 0 & 0 & 0 \end{bmatrix}
\eq
with scalar Lagrange multiplier $\mu$. Combining this as in \eqref{merged} yields the gpH DAE system
\bq
\begin{array}{rcl}
\begin{bmatrix} \dot{q}_1 \\ \dot{q}_2 \\ \dot{p}_1 \\ \dot{p}_2 \\ 0 \\ 0 \end{bmatrix} & = &
 \begin{bmatrix} 0 & 0 & 1 & 0 & 0 & 0 \\ 0 & 0 & 0 & 1 & 0  & 0\\ -1 & 0 & 0 & 0 & 1 & 0\\ 0 & -1 & 0 & 0 & 0 & 0 \\
0 & 0 & -1 & 0 & 0 & 0 \\ 0 & 0 & 0 & 0 & 0 & 1 \end{bmatrix} \cdot
\\
&&\begin{bmatrix} k_1 & -1 & 0 & 0 & 0 & 1 \\ -1 & 0 & 0 & 0 & 0 & -1 \\ 0 & 0 & \frac{1}{m_1} & 0 &  0 & 0\\ 0 & 0 & 0 & \frac{1}{m_2}  & 0 & 0 \\
0 & 0 & 0 & 0 & 1 & 0 \\
1 & -1 & 0 & 0 & 0 & 0 \end{bmatrix}
\begin{bmatrix} q_1 \\ q_2 \\ p_1 \\ p_2 \\ \lambda \\ \mu \end{bmatrix} 
\end{array}
\eq
\end{example}

\section{Conclusions}
%
Following \cite{barbero}, and inspired by \cite{bmxz}, we have given a geometric definition
of generalized port-Hamiltonian DAE systems, defined by pairs of Dirac structures
and Lagrangian subspaces. For physical models, the Dirac structure
corresponds to the interconnection structure of the system, while the Lagrangian subspace corresponds to the definition of their energy. This
generalizes the classical definition of port-Hamiltonian systems by symmetrizing
the role of energy and co-energy variables and allowing for degenerate energy or co-energy functions. In particular we analyzed their algebraic
constraints and representations as DAE systems using the kernel or image representations of both the
Dirac structure and the Lagrangian subspace. The further study
of this class of structured DAE systems, e.g. regularity and index
analysis, seems of great interest. Although for clarity of exposition we restricted attention to systems without energy-dissipation and external variables, the extension to generalized port-Hamiltonian systems {\it with} energy-dissipation and external
variables is straightforward by replacing the Hamiltonian function with a Lagrangian subspace. Important extensions concern the generalization to {\it nonlinear}
systems, replacing Lagrangian subspaces by Lagrangian
submanifolds (see \cite{barbero}), and to {\it distributed-parameter} systems.

\medskip

\noindent
{\bf Acknowledgements} The authors would like to thank Volker Mehrmann (TU Berlin) and Hans Zwart (University of Twente) for stimulating discussions on port-Hamiltonian DAEs.

\section{Appendix}

\begin{proposition}
\label{A}
A subspace $\L \subset \X \times \X^*$ with $\dim \X=n$ is a Lagrangian subspace if and only if there exist $n \times n$ matrices $P,S$ satisfying
\bq
\label{sp}
S^TP = P^TS , \; \rank \begin{bmatrix} S^T & P^T \end{bmatrix} =n
\eq
such that (see \eqref{image})
\bq
\label{image1}
\L=\{(x,e) \in \X \times \X^* \mid \exists z \in \Z=\mR^n \mbox{ s.t. }\begin{bmatrix} x \\ e \end{bmatrix} = \begin{bmatrix} P \\ S \end{bmatrix}z \}
\eq
\end{proposition}
\begin{proof}
The 'if' direction follows by checking that $\langle (x_1,e_1), (x_2,e_2) \rangle_-=0$ for any two pairs $(x_i,e_i)$ with $x_i=Pz_i, e_i=Sz_i, i=1,2$, and $P,S$ satisfying \eqref{sp}.\\
For the 'only if' direction we note that any $n$-dimensional subspace $\L$ can be written as in \eqref{image1} for certain $n \times n$ matrices $P,S$ satisfying $\rank \begin{bmatrix} S^T & P^T \end{bmatrix} =n$. Then take any two pairs $(x_i,e_i) \in \L$ with $x_i=Pz_i, e_i=Sz_i, i=1,2$. Since $\L$ is Lagrangian it follows that
\bq
\begin{array}{rcl}
0 &= &\langle (x_1,e_1), (x_2,e_2) \rangle_-= z_2^TS^TPz_1 - z_1^TS^TPz_2 =\\[2mm]
&& -z_1^T(S^TP - P^TS)z_2
 \end{array}
\eq
for all $z_1,z_2$, implying that $S^TP = P^TS$.
\end{proof}

\begin{proposition}
\label{C}
Consider any Lagrangian subspace $\L \subset \X \times X^*$ with kernel representation (see the previous Proposition \ref{A})
\bq
\L=\{(x,e) \in \X \times \X^* \mid S^Tx - P^Te=0 \}
\eq
for $n \times n$ matrices $P,S$ satisfying \eqref{sp}. Suppose $\rank P=m \leq n=\dim \X$.
Then there exists an $m$-dimensional sub-vector $x_1$ of $x \in \X$, and a complementary $n-m$-dimensional sub-vector $e_2 \in \X^*$ such that $\L$ is represented as
\bq
\L=\{(x,e) \in \X \times \X^* \mid \begin{bmatrix} e_1 \\ x_2 \end{bmatrix} = \widehat{Q} \begin{bmatrix} x_1 \\ e_2 \end{bmatrix} \}
\eq
with
\bq
\label{signature}
\widehat{Q}^T \begin{bmatrix} I_m & 0 \\ 0 & - I_{n-m} \end{bmatrix}= \begin{bmatrix} I_m & 0 \\ 0 & - I_{n-m} \end{bmatrix}\widehat{Q}\eq
\end{proposition}
\begin{proof}
The proof resembles the proof of a similar statement for Dirac structures in \cite{blochcrouch}.
Write, possibly after row permutations of $P$, $P^T = \begin{bmatrix} P^T_1 & P^T_2 \end{bmatrix}$ with $P_1$ having $m$ rows and $\rank P_1 = \rank P$. Then $\im P^T_2 \subset \im P^T_1$. Furthermore, $S^TP = P^TS$ yields
\bq
\label{sp12}
S^T_1P_1 + S^T_2P_2 = P^T_1S_1 + P^T_2S_2
\eq
Combined with surjectivity of $P_1$ and $\im P^T_2 \subset \im P^T_1$ this yields $\im S^T_1 \subset \im P^T_1 + \im S^T_2$. Hence
\bq
\rank \begin{bmatrix} S^T_2 & P^T_1 \end{bmatrix} = \rank \begin{bmatrix} S^T_1 & S^T_2 & P^T_1 & P^T_2 \end{bmatrix} =n,
\eq
thus implying that $\begin{bmatrix} S^T_2 & P^T_1 \end{bmatrix}$ is invertible. In view of \eqref{image1} we have
\bq
\begin{bmatrix} e_1 \\ x_2 \end{bmatrix} = \begin{bmatrix} S_1 \\ P_2 \end{bmatrix}z, \; 
\begin{bmatrix} x_1 \\ e_2 \end{bmatrix} = \begin{bmatrix} P_1 \\ S_2 \end{bmatrix}z
\eq
implying that
\bq
\begin{bmatrix} e_1 \\ x_2 \end{bmatrix} = \begin{bmatrix} S_1 \\ P_2 \end{bmatrix} \left(\begin{bmatrix} P_1 \\ S_2 \end{bmatrix}\right)^{-1}\begin{bmatrix} x_1 \\ e_2 \end{bmatrix} =: \widehat{Q} \begin{bmatrix} x_1 \\ e_2\end{bmatrix}
\eq
Since $\L$ is Lagrangian it follows that for all $(x_j,e_j) \in \L, j=a,b$
\bq
x^T_be_a= x^T_a e_b
\eq
Writing out $x_j= \begin{bmatrix} x_{j1} \\ x_{j2} \end{bmatrix}$ and $e_j= \begin{bmatrix} e_{j1} \\ e_{j2} \end{bmatrix}$ this yields
\bq
x^T_{b1}e_{a1} -x^T_{a2}e_{b2} = x^T_{a1}e_{b1} - x^T_{b2}e_{a2}
\eq
implying equality \eqref{signature}.

\end{proof}

\begin{proposition}
\label{B}
Let $\L \subset \X \times \X^*$ be a Lagrangian subspace. Then 
\bq
\label{bilinear}
\Sigma (\pi(v),\pi(w)) := <\pi^*(v) \mid \pi (w)>, \quad v,w \in \L
\eq
is a well-defined and symmetric bilinear form on $\pi(\L)$. Furthermore, the symmetric map induced by $\Sigma$ can be extended to the symmetric map $\widetilde{Q}$ as in \eqref{extsym} with $\ker M^T=\pi(\L)$, in such a way that $\L$ given by \eqref{Lagrangeconstraint}.
\end{proposition}
\begin{proof}
In order to prove that $\Sigma$ is {\it well-defined} let $v_1,v_2$ be such that $\pi(v_1)=\pi(v_2)$. Then $v:=v_1 - v_2 \in \L$ satisfies $\pi(v)=0$, and thus for any $w \in \L$
\bq
<\pi^*(v) \mid \pi(w) > = <\pi^*(w) \mid \pi(v) > =0
\eq
showing that indeed $<\pi^*(v_1) \mid \pi (w)> = <\pi^*(v_2) \mid \pi (w)>$ for any $w \in \L$. Symmetry of $\Sigma$ directly follows from
$<\pi^*(v) \mid \pi(w) > = <\pi^*(w) \mid \pi(v) >$ for any two $v,w \in \L$. As done in \cite{dalsmo} for the Dirac structure case we may extend the symmetric map induced by $\Sigma$ to the symmetric map $Q$ as in the left-upper block of \eqref{extsym}. Since $\L$ is Lagrangian it easily follows that $\L \cap (0 \times \X^*)= \pi(\L)^{\perp}$ with ${}^{\perp}$ denoting the orthogonal complement with respect to the duality pairing between $\X$ and $\X^*$. Define $M$ such that $\ker M^T=\pi(\L)$. Now, let $(x,e) \in \L$. Then $x \in \ker M^T=\pi(\L)$ and $e=Qx$ modulo $(\ker M^T)^{\perp}= \im M$, and thus $\L$ is indeed given by \eqref{Lagrangeconstraint}.
\end{proof}

\end{document}